\documentclass[11pt]{amsart} 
\usepackage{amsmath, amssymb, latexsym, tikz}
\usepackage{caption,ifthen,url,fancyhdr,cite}
\usepackage[foot]{amsaddr}    
\usepackage{libertine}     
\usepackage{listings} 

\newcommand\textcode[1]{{\footnotesize \rm\texttt{#1}}}   

\makeatletter 
\renewcommand\subsubsection{\@startsection{subsubsection}{3}%
  \z@{.5\linespacing\@plus.7\linespacing}{-.5em}%
  {\normalfont\bfseries}} 
\makeatother
 
\usepackage[T1]{fontenc} 

\newcommand\makequotestraight{%
\begingroup\lccode`~=`' 
\lowercase{\endgroup\let~}\textquotesingle
\catcode`'=\active
}

\usepackage[a4paper,right=2.5cm, left=2.5cm, top=2.5cm, bottom=2.5cm, marginpar=1.6cm]{geometry}

\theoremstyle{mythm}      
\newtheorem{theorem}{Theorem}[section]

\newtheorem{lemma}[theorem]{Lemma}

\theoremstyle{mydef}

\theoremstyle{myrem}

\numberwithin{equation}{section}

\newcounter{ithmcount}

\renewcommand{\leq}{\leqslant} 

\newcommand{\PSL}{{\rm PSL}}

\newcommand{\MM}{\mathbf{M}}

\newcommand{\GG}{\mathbf{G}}

\newcommand{\Gx}{2^{1+24}\udot{\rm Co}_1}

\newcommand{\mt}[1]{{\rm #1}}

\makeatletter
\newcommand{\udot}{\mathpalette\udot@\relax}
\newcommand{\udot@}[2]{%
  \begingroup
  \sbox\z@{$#1{:}$}%
  \sbox\tw@{$#1{.}$}%
  \raisebox{\dimexpr\ht\z@-\ht\tw@}{$\m@th#1.$}%
  \endgroup
}
\makeatother

\usepackage{fancyvrb}
\fvset{commandchars=\\\{\}}

\makeatletter
\@namedef{subjclassname@1991}{2020 MSC}
\makeatother

\begin{document}

\title{Conjugacy class fusion from four maximal subgroups of the Monster}
\subjclass[2000]{} 
\author[A. Pisani]{Anthony Pisani}
\author[T. Popiel]{Tomasz Popiel}
\address[Pisani, Popiel]{School of Mathematics, Monash University, Clayton VIC 3800, Australia.}
\email{\rm apis0001@student.monash.edu, tomasz.popiel@monash.edu}
\thanks{{\em Acknowledgements.} We thank Thomas Breuer for bringing the problem to our attention, and for valuable feedback which helped to improve the paper. 
We also thank Heiko Dietrich and Melissa Lee for helpful conversations related to the paper.}
\date{\today}

\begin{abstract}
We determine the conjugacy class fusion from certain maximal subgroups of the Monster to the Monster, to justify the addition of these data to the Character Table Library in the computational algebra system \textsf{GAP}.
The maximal subgroups in question are $(\PSL_2(11) {\times} \PSL_2(11)){:}4$, $11^2{:}(5 {\times} 2\text{A}_5)$, $7^2{:}\text{SL}_2(7)$, and $\PSL_2(19){:}2$. 
Our proofs are supported by reproducible calculations carried out using the Python package \textcode{mmgroup}, a computational construction of the Monster recently developed by Seysen.
\end{abstract}

\maketitle

\thispagestyle{empty}

\section{Introduction}\label{sec_intro} 

\noindent The Monster, $\MM$, is the largest of the $26$ sporadic finite simple groups. 
Computing with $\MM$ is notoriously difficult, owing to a lack of sufficiently small representations that would allow $\MM$ to be constructed in a standard computational algebra system such as \textsf{GAP} \cite{gap}. 
The Python package \textcode{mmgroup} developed by Seysen \cite{sey20,sey22,sey_python,sey_python_docs} addresses this issue by providing a stand-alone construction of $\MM$ in which the group multiplication can be performed efficiently. 
With the aid of \textcode{mmgroup}, it is now possible, at least in principle, to tackle various open problems concerning $\MM$ that seem unamenable to theoretical methods. 
We say ``in principle'' because \text{mmgroup} (understandably) currently lacks much of the functionality for computing with groups that one is accustomed to having at hand in e.g. \textsf{GAP}. 
In particular, there is no easy way to construct centralisers or normalisers of arbitrary elements or subgroups, nor to solve the conjugacy problem for arbitrary conjugacy classes. 
The first major application of \textcode{mmgroup} is the recent completion of the classification of the maximal subgroups of $\MM$ by Dietrich, Lee, and the second author \cite{ourpaper}. 
Evidently, the current limitations of \textcode{mmgroup} render the task of constructing maximal subgroups of $\MM$ in \textcode{mmgroup} highly non-trivial; we refer the reader to e.g. \cite[Remarks 3.6, 4.4, and 7.4]{ourpaper} for discussions of some techniques that can be used to address this problem. 

The purpose of this paper is to construct, in \textcode{mmgroup}, copies of the maximal subgroups 
$$(\PSL_2(11) {\times} \PSL_2(11)){:}4, \; 11^2{:}(5 {\times} 2\mt{A}_5), \; 7^2{:}\text{SL}_2(7), \; \text{and } \; \PSL_2(19){:}2$$
of $\MM$, and to determine the conjugacy class fusion from each subgroup to $\MM$. 
That is, for each conjugacy class $\mathcal{C}$ of each subgroup, we determine the conjugacy class of $\MM$ that contains~$\mathcal{C}$. 
These data, which are presented in Theorems \ref{thm:L2(11)}, \ref{thm:Sylow11}, \ref{thm:7B2}, and \ref{thm:PGL2(19)}, were added to the \textsf{GAP} Character Table Library~\cite{GAPbc} (v.~1.3.8) based on an earlier draft of this paper.  
(The Character Table Library itself is an indispensable resource with myriad applications; see e.g. \cite{CTLdoc2}.) 
With one exception, the fusion question has been answered for all other maximal subgroups of $\MM$; in particular, $\PSL_2(29){:}2$ is handled in the file ``other\_gens'' in the GitHub repository \cite{ourfile} associated with the paper \cite{ourpaper}. 
The exception is the $2$-local maximal subgroup $2^{5+10+20}.(\mt{S}_3 {\times} \PSL_5(2))$. 
The issue in this case is that the character table of the group is not known. 
We do not attempt to address this case here, but we note that the first author has constructed a copy of $2^{5+10+20}.(\mt{S}_3 {\times} \PSL_5(2))$ in \textcode{mmgroup}, so it should now be possible to produce a `small' representation of the group and then calculate the character table computationally. 

Generating sets for the aforementioned maximal subgroups of $\MM$ are presented in Lemmas \ref{propL211xL211}, \ref{prop11^2}, \ref{prop7^2}, and \ref{propL219}. 
In the interest of brevity, we make no attempt to explain how we found these generating sets; as noted above, such searches are, typically, very difficult to carry out in \textcode{mmgroup}. 
We also note that the first author is currently working, as part of his Honours project at Monash University, to construct all of the maximal subgroups of $\MM$ that are not considered either here or in \cite{ourpaper,ourfile}.


\section{Preliminaries} \label{secNotation}

Notation mostly follows the Atlas \cite{atlas}, but we denote a projective special linear group by $\PSL_n(q)$ instead of $\text{L}_n(q)$. 
We use $n$ to denote a cyclic group of order $n$, and $\mt{A}_n$ and $\mt{S}_n$ to denote the alternating and symmetric groups of degree $n$. 
An extension of a group $B$ by a group $A$ is denoted by either $A.B$ or $AB$, where $A$ is the normal subgroup. 
If an extension is known to be split (respectively, non-split) and we wish to emphasise this, then we write $A{:}B$ (respectively, $A\udot B$). 
An elementary abelian group of order $p^k$ is denoted by $p^k$, and $p^{k+\ell}$ means $p^k.p^\ell$, for $p$ a prime and $k, \ell \in \mathbb{Z}^+$. 
Notation for conjugacy classes follows the Atlas and the \textsf{GAP} Character Table Library. 

Those parts of our proofs that rely on \textcode{mmgroup} are reproducible using only very basic functionality of \textcode{mmgroup}, which we now outline. 
For a broader introduction to \textcode{mmgroup}, we refer the reader to the \textcode{mmgroup} documentation \cite{sey_python_docs}. 
Basic group operations are performed using the `obvious' Python syntax: multiplication is done via \textcode{*}, powering and conjugation are both done via \textcode{**}, and inversion corresponds to powering by $-1$. 
Less obviously, \textcode{mmgroup} implements a distinguished conjugate $\GG < \MM$, called $G_{x0}$ in the documentation, of the $2\text{B}$-involution centraliser $\Gx$. 
Specifically, $\GG$ is the centraliser of the $2\text{B}$-involution $z \in \MM$ with \textcode{mmgroup} label $\textcode{M<x\_1000h>}$. 
Membership in $\GG$ can be tested via the method \textcode{in\_G\_x0}. 
If $g \in \GG$, then the method \textcode{chi\_G\_x0} can be used to calculate $\chi_\MM(g)$, where $\chi_\MM$ is the character of the $196883$-dimensional irreducible $\mathbb{C}\MM$-module. 
Specifically, $\chi_\MM(g)$ is given by \textcode{g.chi\_G\_x0()[0]}. 
We stress that \textcode{chi\_G\_x0} cannot be applied to elements outside of $\GG$.
The order of $g \in \MM$ is given by \textcode{g.order()}. 
If $g$ has order $2$, then the method \textcode{conjugate\_involution} returns an integer $1$ or $2$ according to whether $g \in 2\text{A}$ or $2\text{B}$, and an element $c \in \MM$ that conjugates $g$ to a fixed representative of $g^\MM$. 
If $g \in 2\text{B}$, then $c$ conjugates $g$ to $z \in Z(\GG)$. 
This allows us to conjugate an arbitrary element $h \in C_\MM(g)$ into $\GG$, and thereby calculate $\chi_\MM(h)$. 

The \textcode{mmgroup} calculations described in our proofs are fully documented in an accompanying Python file \cite{ourfile_new}. 
We frequently refer to various code listings containing several elements of $\MM$ in \textcode{mmgroup} format; these are collected in Appendix~\ref{app:listings}.
Note that we often use subscripts to indicate element orders, e.g. the element $x_2 \in \MM$ in Listing~\ref{fig:secL211} has order $2$.


\section{The maximal subgroup $(\PSL_2(11) {\times} \PSL_2(11)){:}4$} \label{secL211}

According to the Atlas \cite[p.~234]{atlas}, a maximal subgroup of $\MM$ of shape $(\PSL_2(11) {\times} \PSL_2(11)){:}4$ is the normaliser of a direct product of two copies of $\PSL_2(11)$, both of which are of ``type AAA'', meaning that all of their elements of orders $2$, $3$, and $5$ lie in the conjugacy classes $2\text{A}$, $3\text{A}$, and $5\text{A}$ of~$\MM$. 
It is well known that $\MM$ has a unique conjugacy class of subgroups $\PSL_2(11) {\times} \PSL_2(11)$ in which both direct factors are of type AAA, but we give a proof for the sake of clarity. 
By \cite[Theorem~19]{NW02}, there is a unique $\MM$-class of subgroups $\PSL_2(11)$ of type AAA. 
If $P$ is such a subgroup, then, by \cite[Table~5]{N98}, its centraliser is isomorphic to $\text{M}_{12}$. 
There are two classes of $\PSL_2(11)$ in $\text{M}_{12}$: the maximal ones, and those contained in a maximal $\text{M}_{11}$. 
In $C_\MM(P) \cong \text{M}_{12}$, the latter are not of type AAA, because the $\text{M}_{11}$ that contain them are not of type AAA; indeed, by \cite[Table~5]{N98}, the centraliser of a type-AAA $\text{M}_{11} < \MM$ has shape $\mt{S}_6.2$, so such an $\text{M}_{11}$ does not centralise any $\PSL_2(11)$.

\begin{lemma} \label{propL211xL211}
The elements $x_2,x_{11},x_4 \in \MM$ in Listing \ref{fig:secL211} generate a maximal subgroup of \hspace{0.01pt} $\MM$ of shape $(\PSL_2(11) {\times} \PSL_2(11)){:}4$. 
\end{lemma}

\begin{proof}
As previously noted, the calculations described here and throughout the paper are documented in an accompanying Python file \cite{ourfile_new}. 
Let $P = \langle x_2,x_{11} \rangle$. 
We first verify that $a=x_2$ and $b=x_{11}$ satisfy the presentation $\langle a,b \mid a^2 = b^{11} = (ba)^3 = (b^2ab^6a)^3 = 1 \rangle$ for $\PSL_2(11)$ given in \cite[Theorem~A]{BM68}. 
Because $\PSL_2(11)$ is simple, it follows from Von Dyck's Theorem that $P \cong \PSL_2(11)$.
A further calculation shows that $x_2$ and $x_{11}$ lie in the distinguished $2\text{B}$-centraliser $\GG<\MM$, so we can evaluate the character $\chi_\MM$ of the $196883$-dimensional irreducible $\mathbb{C}\MM$-module on all elements of $P$, as explained in Section~\ref{secNotation}. 
The character table of $\MM$ shows that $\chi_\MM(g)$ determines the $\MM$-class of an element $g$ of order $i$ for each $i \in \{2,3,5\}$, so we readily verify that $P$ contains elements of classes $2\text{A}$, $3\text{A}$, and $5\text{A}$. 
(Elements of orders $2$ and $3$ are exhibited in the above presentation, and e.g. $x_2x_{11}^3$ has order~$5$.) 
We then check that $P$ commutes with but is not equal to its conjugate under $x_4$. 
To verify the latter, it suffices to check that $x_{11}^{x_4} \not \in \langle x_{11} \rangle$, because $\langle x_{11} \rangle = C_P(x_{11})$. 
The group $A = \langle P,P^{x_4} \rangle$ is therefore isomorphic to $\PSL_2(11) {\times} \PSL_2(11)$, and it follows from the discussion above that $N_\MM(A)$ is a maximal subgroup of the desired type. 
We now claim that $x_4$ extends $A$ to $N_\MM(A)$. 
We find that $x_4$ has order $4$, and that $x_4^2$ normalises $P$, so $x_4$ normalises $A$ (but lies outside of it). 
More specifically, $x_4^2$ centralises $x_2$ and inverts $x_{11}$, so $x_4^2$ does not lie in $A$,  because if it did then it would project to $N_P(\langle x_{11} \rangle) \cong 11{:}5$, which has odd order. 
Therefore, $B=\langle x_4 \rangle$ intersects $A$ trivially, so $AB = N_\MM(A)$, as claimed. 
\end{proof}

\begin{theorem} \label{thm:L2(11)}
If $S$ is a maximal subgroup of \hspace{0.01pt} $\MM$ of shape $(\PSL_2(11) {\times} \PSL_2(11)){:}4$, then the conjugacy classes of $S$ fuse in $\MM$ as indicated in Table~\ref{table1}.
\end{theorem}

\begin{proof}
Take $S$ to be the subgroup $\langle x_2,x_{11},x_4 \rangle$ of $\MM$ constructed in Lemma~\ref{propL211xL211}. 
We first apply the \textsf{GAP} function \textsf{PossibleClassFusions} to the character tables of $S$ and $\MM$ to obtain a list of possible class fusions. 
The character tables of $S$ and $\MM$ are stored in \textsf{GAP} as \textsf{CharacterTable("(L2(11)xL2(11)):4")} and \textsf{CharacterTable("M")}, respectively. 
Let us label the $S$-classes according to their labelling in the former, using lowercase letters to distinguish them from $\MM$-classes, which we label with uppercase letters as before. 
There are two possible class fusions, distinguished by whether the $S$-classes $20\text{c--f}$ fuse to the $\MM$-class $20\text{E}$ or $20\text{F}$, so it suffices to determine the $\MM$-class of one element in the union of $20\text{c--f}$, i.e. \textsf{PossibleClassFusions} then determines how the remaining $S$-classes fuse in $\MM$. 
Note that $S$ has exactly two other classes of elements of order $20$, labelled $20\text{a}$ and $20\text{b}$. 
(The elements in $20\text{a--b}$ are powers of elements of order $60$ in $S$, and the elements in $20\text{c--f}$ are not.) 
The list of possible class fusions indicates that $20\text{a--b}$ fuse to $20\text{B}$. 
According to the power maps in the character table of $\MM$, elements in $20\text{B}$ power to $2\text{A}$. 
On the other hand, elements in $20\text{E}$ and $20\text{F}$ power to $2\text{B}$, so every element of order $20$ in $S$ that powers to $2\text{B}$ lies in one of $20\text{c--f}$. 
We find that one such element is $x_{20} = x_{11}x_2x_4$ and then proceed as described in Section~\ref{secNotation} to calculate $\chi_\MM(x_{20}) = 2$, which indicates that $x_{20} \in 20\text{E}$. 
\end{proof}

\begin{table}[t]
\begin{tabular}{llllll}
\hline
$2\text{a} \rightarrow 2\text{A}$ & $4\text{b} \rightarrow 4\text{B}$ & $6\text{b} \rightarrow 6\text{E}$ & $10\text{cd} \rightarrow 10\text{D}$ & $12\text{fg} \rightarrow 12\text{E}$ & $24\text{abcd} \rightarrow 24\text{H}$ \\
$2\text{bc} \rightarrow 2\text{B}$ & $4\text{cd} \rightarrow 4\text{D}$ & $6\text{c} \rightarrow 6\text{C}$ & $10\text{ef} \rightarrow 10\text{B}$ & $15\text{ab} \rightarrow 15\text{A}$ & $30\text{ab} \rightarrow 30\text{B}$ \\
$3\text{a} \rightarrow 3\text{A}$ & $5\text{ab} \rightarrow 5\text{A}$ & $6\text{e} \rightarrow 6\text{D}$ & $10\text{g} \rightarrow 10\text{E}$ & $20\text{ab} \rightarrow 20\text{B}$ & $33\text{a} \rightarrow 33\text{B}$ \\
$3\text{b} \rightarrow 3\text{B}$ & $5\text{cde} \rightarrow 5\text{B}$ & $8\text{ab} \rightarrow 8\text{D}$ & $12\text{abe} \rightarrow 12\text{H}$ & $20\text{cdef} \rightarrow 20\text{E}$ & $60\text{abcd} \rightarrow 60\text{A}$ \\
$4\text{a} \rightarrow 4\text{C}$ & $6\text{ad} \rightarrow 6\text{A}$ & $10\text{ab} \rightarrow 10\text{A}$ & $12\text{cd} \rightarrow 12\text{C}$ & $22\text{a} \rightarrow 22\text{A}$ & $66\text{a} \rightarrow 66\text{A}$ \\
\hline
\end{tabular}
\caption{Data for Theorem~\ref{thm:L2(11)}. 
Conjugacy class fusion from a maximal subgroup $S$ of $\MM$ of shape $(\PSL_2(11) {\times} \PSL_2(11)){:}4$ to $\MM$. 
The conjugacy classes of $S$ are labelled as in \textsf{CharacterTable("(L2(11)xL2(11)):4")} in \textsf{GAP}. 
Lowercase letters are used for $S$-classes, uppercase letters are used for $\MM$-classes, and e.g. $2\text{bc} \rightarrow 2\text{B}$ means that the $S$-classes $2\text{b}$ and $2\text{c}$ fuse to the $\MM$-class $2\text{B}$.
Elements of orders $11$ and $55$ (and $1$) are omitted from the table because $\MM$ has unique classes of elements of these orders.
}
\label{table1}
\end{table}


\section{The maximal subgroup $11^2{:}(5 {\times} 2\mt{A}_5)$} \label{sec11^2}

A maximal subgroup of $\MM$ of shape $11^2{:}(5 {\times} 2\mt{A}_5)$ is the normaliser of a Sylow $11$-subgroup. 

\begin{lemma} \label{prop11^2}
The elements $x_{11},y_{11},x_3,x_4,x_5 \in \MM$ in Listing~\ref{fig:sec11^2} generate a maximal subgroup of \hspace{0.01pt} $\MM$ of shape $11^2{:}(5 {\times} 2\mt{A}_5)$. 
\end{lemma}

\begin{proof}
We first check that $x_{11}$ and $y_{11}$ have order $11$, commute, and are not powers of each other, so that $A = \langle x_{11},y_{11} \rangle$ is a Sylow $11$-subgroup of $\MM$. 
A further calculation shows that $B = \langle x_3,x_4,x_5 \rangle$ normalises $A$, and that $x_5$ commutes with $x_3$ and $x_4$. 
Moreover, $x_2 = x_4^2$ is central in $\langle x_3,x_4 \rangle$, and $a=x_3$ and $b=x_4$ satisfy the well-known presentation $\langle a,b \mid a^2 = b^2 = (ab)^5 = 1 \rangle$ for $\mt{A}_5$ modulo $\langle x_2 \rangle$. 
Therefore, $B \cong 5 {\times} 2\mt{A}_5$. 
Given that $B \cap A$ is trivial, it follows that $AB = N_\MM(A)$. 
\end{proof}

\begin{theorem} \label{thm:Sylow11}
If $S$ is a maximal subgroup of \hspace{0.01pt} $\MM$ of shape $11^2{:}(5 {\times} 2\mt{A}_5)$, then the conjugacy classes of $S$ fuse in $\MM$ as indicated in Table~\ref{table2}.
\end{theorem}

\begin{proof}
Take $S$ to be the subgroup $\langle x_{11},y_{11},x_3,x_4,x_5 \rangle$ of $\MM$ constructed in Lemma~\ref{prop11^2}. 
We proceed as in the proof of Theorem~\ref{thm:L2(11)}, applying the \textsf{GAP} function \textsf{PossibleClassFusions} to simplify our calculations. 
The character table of $S$ is stored in \textsf{GAP} as \textsf{CharacterTable("11\^{}2:(5x2.A5)")}. 
There are seven possible class fusions. 
In each case, the elements of order $30$ in $S$ are seen to lie in a single $\MM$-class, although this is clear {\em a priori} because $S$ has a unique conjugacy class of cyclic subgroups of order $30$ and all $\MM$-classes of elements of order $30$ are rational, i.e. every $30 < \MM$ meets a single class of elements of order $30$. 
The seven possibilities are reduced to two upon showing that the elements of order $30$ in $S$ lie in the $\MM$-class $30\text{E}$. 
We find that $x_{30} = x_5x_3x_4^2$ has order $30$ and powers to $2\text{B}$, which allows us to calculate $\chi_\MM(x_{30}) = -1$. 
This implies that $x_{30}$ lies in $30\text{C}$ or $30\text{E}$. 
The former case is impossible according to the list of possible class fusions, so $x_{30} \in 30\text{E}$, but we also confirm this by calculating $\chi_\MM(x_{30}^3) = 20$, which implies that $x_{30}^3 \in 10\text{D}$ and hence $x_{30} \not \in 30\text{C}$, because $30\text{C}$-elements power to $10\text{A}$. 
The remaining two possible class fusions are distinguished by the existence or non-existence in $S$ of $10\text{E}$-elements. 
The element $x_{10} = x_4(x_3x_5)^2$ has order $10$ and powers to $2\text{B}$; we calculate $\chi_\MM(x_{10}) = 0$, which indicates that $x_{10} \in 10\text{E}$. 
The other data in Table~\ref{table2} are then obtained via \textsf{PossibleClassFusions}.
\end{proof}

\begin{table}[t]
\begin{tabular}{llll}
\hline
$2\text{a} \rightarrow 2\text{B}$ & $5\text{abcdefhjln} \rightarrow 5\text{B}$ & $10\text{abcdej} \rightarrow 10\text{D}$ & $15\text{abcd} \rightarrow 15\text{D}$ \\
$3\text{a} \rightarrow 3\text{C}$ & $5\text{gikm} \rightarrow 5\text{A}$ & $10\text{fhln} \rightarrow 10\text{E}$ & $20\text{abcd} \rightarrow 20\text{E}$ \\
$4\text{a} \rightarrow 4\text{D}$ & $6\text{a} \rightarrow 6\text{F}$ & $10\text{gikm} \rightarrow 10\text{B}$ & $30\text{abcd} \rightarrow 30\text{E}$ \\
\hline
\end{tabular}
\caption{Data for Theorem~\ref{thm:Sylow11}. Conjugacy class fusion from a maximal subgroup $S$ of $\MM$ of shape $11^2{:}(5 {\times} 2\mt{A}_5)$ to $\MM$. 
Notation is as in Table~\ref{table1}, with $S$-classes labelled as in \textsf{CharacterTable("11\^{}2:(5x2.A5)")} in \textsf{GAP}. 
Elements of orders $1$, $11$, and $55$ are omitted.
}
\label{table2}
\end{table}


\section{The maximal subgroup $7^2{:}\text{SL}_2(7)$} \label{sec7^2}

Recall that $\MM$ has exactly two conjugacy classes of elements of order $7$, denoted $7\text{A}$ and $7\text{B}$. 
Per \cite{Ho88} and \cite[Section~10]{W88}, a maximal subgroup of $\MM$ of shape $7^2{:}\text{SL}_2(7)$ is the normaliser of a certain type of $7\text{B}$-pure subgroup of shape $7^2$, ``pure'' meaning that the $7^2$ contains no $7\text{A}$-elements. 
If $x \in 7\text{B}$, then $N_\MM(\langle x \rangle)$ has shape $7^{1+4}{:}(3 {\times} 2\mt{S}_7)$. 
To extend $\langle x \rangle$ to a $7^2$ of the correct type, it suffices to adjoin an element $y$ of class $7\text{B}$ that commutes with $x$ and lies outside of the normal subgroup $7^{1+4}$ of $N_\MM(\langle x \rangle)$, such that all non-trivial elements of $\langle x,y \rangle \cong 7^2$ are in class $7\text{B}$. 

\begin{lemma} \label{prop7^2}
The elements $x_7,y_7 \in \MM$ in Listing~\ref{fig:sec7^2} generate a subgroup $A \cong 7^2$ of \hspace{0.01pt} $\MM$ such that $N_\MM(A)$ is a maximal subgroup of shape $7^2{:}\textnormal{SL}_2(7)$. 
The elements $x_4,x_{14} \in \MM$ in Listing~\ref{fig:sec7^2} extend $A$ to $N_\MM(A)$. 
\end{lemma}

\begin{proof}
We check that $x_7$ has order $7$ and lies in $\GG$, and calculate $\chi_\MM(x_7) = 1$ to confirm that $x_7 \in 7\text{B}$. 
Next we claim that the elements $a_7$, $b_7$, $c_7$, and $d_7$ in Listing~\ref{fig:sec7^2} extend $\langle x_7 \rangle$ to the normal subgroup $7^{1+4}$ of $N_\MM(\langle x_7 \rangle)$. 
Let $Q = \langle x_7,a_7,b_7,c_7,d_7 \rangle$. 
We check that $a_7$, $b_7$, $c_7$, and $d_7$ centralise $x_7$, so that $Q \leq N_\MM(\langle x_7 \rangle)$, and that they commute pairwise modulo $\langle x_7 \rangle$, so that $Q / \langle x_7 \rangle$ is elementary abelian. 
To confirm that $Q$ is the whole normal $7^{1+4} < N_\MM(\langle x_7 \rangle)$, it then suffices to show that $|Q|=7^5$, because $N_\MM(\langle x_7 \rangle)$ has a unique subgroup of order $7^5$ whose quotient by $\langle x_7 \rangle$ is elementary abelian. 
The latter claim can be verified by constructing a copy of $N_\MM(\langle x_7 \rangle)$ in \textsf{GAP} as \textsf{AtlasGroup("7\^{}(1+4):(3x2.S7)")}, using the \textsf{AtlasRep} package \cite{atlas-web}. 
We check that $a_7$ and $b_7$ commute and are not powers of each other, and that $x_7 \not \in \langle a_7,b_7 \rangle$, so that $\langle x_7,a_7,b_7 \rangle$ is elementary abelian of order $7^3$. 
The elements $c_7$ and $d_7$ also commute, and we check that $\langle c_7,d_7 \rangle$ intersects $\langle x_7,a_7,b_7 \rangle$ trivially, so $|Q|=7^5$. 
The element $y_7$ has order $7$ and centralises $x_7$. 
To confirm that $y_7$ does not lie in $Q$, we check that it does not commute with $a_7$ modulo $\langle x_7 \rangle$. 
To complete the proof of the first assertion, it then remains to show that $A = \langle x_7,y_7 \rangle$ is $7\text{B}$-pure. 
Given that $7\text{B}$ is a rational class, it suffices to check that one element of each of the seven cyclic subgroups of $A$ other than $\langle x_7 \rangle$ lies in $7\text{B}$. 
These subgroups are generated by $y_7$ and $y_7^ix_7$ for $1 \leq i \leq 6$, and we find that $x_7$ is conjugate to each of these generators via the following elements of $B = \langle x_4,x_{14} \rangle$, respectively: $(x_{14}x_4x_{14})^2$,\, $(x_{14}x_4x_{14})^2x_{14}^2x_4$,\, $x_{14}x_4x_{14}^4x_4x_{14}$,\, $x_{14}x_4x_{14}^3$,\, $x_4x_{14}(x_{14}x_4)^2$,\, $x_{14}x_4x_{14}^3x_4$,\, and $(x_{14}x_4x_{14})^2x_{14}^3$.

To prove the second assertion, we first check that $B$ normalises $A$, and that $a=x_4$ and $b=x_{14}$ satisfy the presentation
$
\langle a,b \mid (ab)^3a^{-2} = (ab^4ab^4)^2b^7a^4 = 1 \rangle
$
for $\text{SL}_2(7)$ given in \cite[Theorem~4]{CRpres}, noting also that $|x_4|=4$ and $|x_{14}|=14$. 
Given that $\PSL_2(7)$ has no elements of order~$14$, it follows from Von Dyck's Theorem that $B \cong \text{SL}_2(7)$. 
To confirm that $AB = N_\MM(A)$, it remains to show that $B$ intersects $A$ trivially. 
If not, then the kernel of the action of $B$ on $A$ by conjugation must be non-trivial because $A$ is abelian, and so the central involution in $B$ must centralise $A$. 
A final calculation shows, however, that the central involution in $B$ is $x_4^2$, and that $x_4^2$ does not centralise $x_7$. 
\end{proof}

\begin{theorem} \label{thm:7B2}
A maximal subgroup $7^2{:}\textnormal{SL}_2(7)$ of \hspace{0.01pt} $\MM$ intersects precisely the following conjugacy classes of \hspace{0.01pt} $\MM$ non-trivially: $2\textnormal{B}$, $3\textnormal{C}$, $4\textnormal{D}$, $6\textnormal{F}$, $7\textnormal{B}$, $8\textnormal{F}$, and $14\textnormal{C}$. 
\end{theorem}

\begin{proof}
Consider the subgroup $S = \langle x_7,y_7,x_4,x_{14} \rangle$ of $\MM$ constructed in Lemma~\ref{prop7^2}. 
We claim that the unique $S$-classes of elements of orders $4$ and $6$ fuse to the $\MM$-classes $4\text{D}$ and $6\text{F}$, and that all elements of order $7$ in $S$ lie in $7\text{B}$. 
The theorem then follows from the power maps in the character table of $\MM$. 
(Elements in $4\text{C}$ and $6\text{F}$ square to $2\text{B}$ and $3\text{C}$, respectively, so the unique $S$-classes of elements of orders $2$ and $3$ fuse to $2\text{B}$ and $3\text{C}$; the only elements of order $8$ that square to $4\text{C}$ are those in $8\text{F}$; and the only elements of order $14$ that square to $7\text{B}$ are those in $14\text{C}$.) 
The element $c$ in Listing~\ref{fig:sec7^2} conjugates $x_4$ and $x_{14}$ into $\GG$, so we can calculate $\chi_\MM(g)$ for all $g \in \langle x_4,x_{14} \rangle$. 
We find that $\chi_\MM(x_4)=-13$, which indicates that $x_4 \in 4\text{D}$, and that $x_6 = x_4x_{14}$ has order $6$ and $\chi_\MM$-value $-1$, which indicates that $x_6 \in 6\text{F}$. 
There are nine $S$-classes of elements of order $7$. 
Let us label them as in \textsf{CharacterTable("7\^{}2:2psl(2,7)")} in \textsf{GAP}, using lowercase letters.
The non-trivial elements in the normal subgroup $A = \langle x_7,y_7 \rangle$ of $S$ comprise the class $7\text{a}$, which is already known to fuse to $7\text{B}$. 
The classes $7\text{b}$ and $7\text{f}$ consist of the squares of elements of order $14$. 
The cube of a $7\text{b}$-element is a $7\text{f}$-element, and $7\text{B}$ is a rational class, so it suffices to check that $x_{14}$ squares to a $7\text{B}$-element, which we do by calculating $\chi_\MM(x_{14}^2) = 1$. 
The remaining six $S$-classes of elements of order $7$ (labelled $7\text{cdeghi}$) are precisely the $S$-classes whose elements have centralisers of order $7^2$. 
Denote the union of these classes by $U$. 
Upon constructing a copy of $S$ in \textsf{GAP} as \textsf{AtlasGroup("7\^{}2:2psl(2,7)")}, it is easy to check that for every non-trivial $z_7 \in A$ and every $z_{14} \in S$ of order $14$ that does not normalise $\langle z_7 \rangle$, the elements $z_7z_{14}^2$, $z_7^2z_{14}^2$, and $z_7^3z_{14}^2$ and their third powers form a full set of representatives of the $S$-classes in $U$. 
Given that $7\text{B}$ is rational, it therefore suffices to show that $z_7z_{14}^2, z_7^2z_{14}^2, z_7^3z_{14}^2 \in 7\text{B}$ for some choice of $z_7$ and $z_{14}$. 
We choose $z_7 = x_7$ and $z_{14} = x_{14}^{x_4}$, verifying first that $z_{14}$ does not normalise $\langle x_7 \rangle$. 
We then check that the elements $r$, $s$, and $t$ in Listing~\ref{fig:sec7^2} conjugate $z_7z_{14}^2$, $z_7^2z_{14}^2$, and $z_7^3z_{14}^2$, respectively, to $x_7 \in 7\text{B}$. 
\end{proof}


\section{The maximal subgroup $\PSL_2(19){:}2$} \label{secL219}

To handle this case, we first recall from \cite[Table~3]{N98} that $\MM$ has exactly two conjugacy classes of subgroups isomorphic to $\mt{A}_5$ whose non-trivial elements lie in the $\MM$-classes $2\text{B}$, $3\text{B}$, and $5\text{B}$. 
Such a subgroup $A$ is said to be of ``type B'' or ``type T'' according to whether $C_\MM(A)$ is cyclic of order $2$ or isomorphic to $\mt{S}_3$. 
(The nomenclature is explained in \cite{N98}.) 
By \cite[Theorem~20]{NW02}, in every subgroup $S \cong \PSL_2(19){:}2$ of $\MM$, the elements of orders $2$, $3$, and $5$ must belong to the $\MM$-classes $2\text{B}$, $3\text{B}$, and $5\text{B}$. 
In particular, the subgroups of $S$ isomorphic to $\mt{A}_5$, which comprise a single $S$-class, must all be of a single type, B or T. 
By \cite[Theorem~1]{HW08}, there are exactly two classes of $\PSL_2(19){:}2 < \MM$, distinguished by the type of their $\mt{A}_5$ subgroups, and only those with $\mt{A}_5$ subgroups of type B are maximal.  

\begin{lemma} \label{propL219}
The elements $x_2,x_{19} \in \MM$ in Listing~\ref{fig:secL219} generate a maximal  $\PSL_2(19){:}2 < \MM$. 
In particular, $g_2 = (x_{19}^2x_2)^2$ and $g_3 = x_2x_{19}^2x_2x_{19}$ generate a subgroup of \hspace{0.01pt} $\MM$ that is isomorphic to $\mt{A}_5$ and of type B.
\end{lemma}

\begin{proof}
Let $S = \langle x_2,x_{19} \rangle$. 
A short calculation shows that $a=x_2$ and $b=x_{19}$ satisfy the presentation
$
\langle a,b \mid a^2 = b^{19} = (ab^2)^4 = (abab^2)^3 = 1 \rangle
$
for $\PSL_2(19){:}2 = \text{PGL}_2(19)$ given in \cite[Corollary~5]{pres}. 
Because the only non-trivial quotients of $\PSL_2(19){:}2$ are the group itself and a group of order $2$, this confirms that $S \cong \PSL_2(19){:}2$. 
To show that $S$ is maximal in $\MM$, it suffices to exhibit a type-B $\mt{A}_5<S$. 
The elements $g_2 = (x_{19}^2x_2)^2$ and $g_3 = x_2x_{19}^2x_2x_{19}$ coincide with the elements $g_2$ and $g_3$ in the ``type~B'' case of \cite[Listing~8]{ourpaper}. 
It is shown in \cite[Proposition~4.2]{ourpaper} that $A = \langle g_2,g_3 \rangle$ is isomorphic to $\mt{A}_5$ and of type~B, but we also give an alternative proof that is significantly shorter and easier to reproduce. 

The elements $a=g_2$ and $b=g_3$ satisfy the presentation $\langle a,b \mid a^2=b^3=(ab)^5=1 \rangle$ for $\mt{A}_5$, so $A \cong \mt{A}_5$. 
The method \textcode{conjugate\_involution} shows that $g_2 \in 2\text{B}$;  the elements $c$ and $d$ in Listing~\ref{fig:secL219} conjugate $g_3$ and $g_5 = g_2g_3$, respectively, into $\GG$, whence \textcode{chi\_G\_x0} shows that $\chi_\MM(g_3) = 53$ and $\chi_\MM(g_5) = 8$, so $g_3 \in 3\text{B}$ and $g_5 \in 5\text{B}$. 
Per \cite[Section~4]{N98}, a type-T $\mt{A}_5<\MM$ has normaliser $\mt{S}_3 {\times} \mt{S}_5$, and a type-B $\mt{A}_5<\MM$ has normaliser $\mt{A}_5.4$. 
The element $y_4$ in Listing~\ref{fig:secL219} has order $4$ and normalises but does not centralise $A$, and $y_4^2$ centralises $A$, so $y_4$ extends $A$ to a group of shape $\mt{A}_5.4$. 
Given that $\mt{S}_3 {\times} \mt{S}_5$ has no such subgroup, it follows that $\langle A,y_4 \rangle = N_\MM(A)$, so $A$ is of type~B.
\end{proof}

\begin{theorem} \label{thm:PGL2(19)}
A maximal subgroup $\PSL_2(19){:}2$ of \hspace{0.01pt} $\MM$ intersects precisely the following conjugacy classes of \hspace{0.01pt} $\MM$ non-trivially: $2\textnormal{B}$, $3\textnormal{B}$, $4\textnormal{C}$, $5\textnormal{B}$, $6\textnormal{E}$, $9\textnormal{B}$, $10\textnormal{E}$, $18\textnormal{E}$, $19\textnormal{A}$, and $20\textnormal{F}$. 
\end{theorem}

\begin{proof}
Consider the subgroup $S = \langle x_2,x_{19} \rangle$ of $\MM$ constructed in Lemma~\ref{propL219}. 
Note that the character table of $S$ is stored in \textsf{GAP} as \textsf{CharacterTable("L2(19).2")}. 
We claim that all elements of orders $18$ and $20$ in $S$ lie in the $\MM$-classes $18\text{E}$ and $20\text{F}$, respectively. 
The theorem then follows from the power maps in the character tables of $S$ and $\MM$.
(There is a unique $\MM$-class of elements of order $19$; every element of $S$ of order not equal to $19$ is a power of an element of order $18$ or $20$; elements in $18\text{E}$ power to $9\text{B}$, $6\text{E}$, $3\text{B}$, and $2\text{B}$; and elements in $20\text{F}$ power to $10\text{E}$, $5\text{B}$, $4\text{C}$, and $2\text{B}$.)
Given that $S$ has unique conjugacy classes of cyclic subgroups of orders $18$ and $20$, and that all $\MM$-classes of elements of these orders are rational, it suffices to exhibit one element of each of the classes $18\text{E}$ and $20\text{F}$ in $S$. 
We find that $x_{18} = x_2x_{19}^3$ and $x_{20} = x_2x_{19}$ have orders $18$ and $20$, respectively, and that both power to $2\text{B}$, so we can calculate $\chi_\MM(x_{18}) = 5$ and $\chi_\MM(x_{20}) = 4$, which indicates that $x_{18} \in 18\text{E}$ and $x_{20} \in 20\text{F}$. 
\end{proof}



\appendix
\section{Code listings} \label{app:listings}

The code listings referred to throughout the paper are collected here. 
As explained in the \textcode{mmgroup} documentation~\cite{sey_python_docs}, each string of the form \textcode{"M<...>"} uniquely defines an element of $\MM$ in terms of certain natural generators. 
The elements given here are also included in \cite{ourfile_new}.

\vspace{10pt}
{\footnotesize
\begin{lstlisting}[breaklines=true,language=python,captionpos=b,texcl=false,frame=lines,caption={
Generators for the maximal subgroup $(\PSL_2(11) {\times} \PSL_2(11)){:}4$ of $\MM$ discussed in Section~\ref{secL211}, in \textcode{mmgroup} format; see also \cite{ourfile_new}.\\}, label=fig:secL211]  
x2 = MM("M<y_3dch*x_57fh*d_11ch*p_164842291*l_1*p_2640000*l_1*p_935210>")

x11 = MM("M<y_4bdh*x_120h*d_52ch*p_87984372*l_1*p_3840*l_1*p_21360*l_1*p_135360>")

x4 = MM("M<y_4a9h*x_898h*d_1ach*p_74531712*l_2*p_1900800*l_2*p_12614870*l_2*t_2*l_2*p_2344320*l_2*p_31997157*l_1*t_1*l_1*p_2880*l_2*p_21312*l_2*p_10252800*t_1*l_2*p_1900800*l_2*p_932277*t_1*l_1*p_1499520*l_2*p_64121894*t_1*l_2*p_2597760*l_1*p_42706968*t_2*l_2*p_2956800*l_1*p_42667409>")
\end{lstlisting}}

{\footnotesize
\begin{lstlisting}[breaklines=true,language=python,captionpos=b,texcl=false,frame=lines,caption={
Generators for the maximal subgroup $11^2{:}(5{\times}2\mt{A}_5)$ of $\MM$ discussed in Section~\ref{sec11^2}; see also \cite{ourfile_new}. (Note that $x_{11}$ is the same as in Listing \ref{fig:secL211}, but $x_4$ is not.) \\}, label=fig:sec11^2]  
x11 = MM("M<y_4bdh*x_120h*d_52ch*p_87984372*l_1*p_3840*l_1*p_21360*l_1*p_135360>")

y11 = MM("M<y_389h*x_0d8dh*d_0d9ch*p_150523146*l_1*p_2640000*l_1*p_10668793*t_1*l_2*p_2597760*l_1*p_21348617*t_1*l_1*p_2832000*t_1*l_1*p_1499520*l_2*p_63997801*t_1*l_1*p_1499520*l_2*p_1527956>")

x3 = MM("M<y_479h*x_474h*d_0ad8h*p_170818001*l_1*p_2999040*l_1*p_32071331*l_1*t_1*l_2*p_1920*l_1*p_1394256*l_1*t_2*l_2*p_1900800*l_2*p_21819090*l_1*t_1*l_2*p_2386560*l_2*p_10777289*t_1*l_1*p_1499520*l_2*p_42755921*t_1*l_2*p_2386560*l_2*p_42799217*t_2*l_2*p_1985280*l_1*p_53381138>")

x4 = MM("M<y_3d7h*x_1c77h*d_206h*p_198338203*l_1*p_2027520*l_1*p_22753408*l_2*t_1*l_2*p_467520*l_2*p_22260720*l_2*t_1*l_1*p_1415040*l_1*p_10666848*l_1*p_10539840*t_1*l_1*p_1933440*l_1*t_2*l_2*p_46168320*l_2*t_1*l_1*p_2999040*l_1*p_16357*t_1*l_2*p_2597760*l_1*p_86256657*t_1*l_2*p_2386560*l_2*p_21424635>")

x5 = MM("M<y_47ah*x_1f1eh*d_482h*p_238402577*l_2*p_1900800*l_2*p_33456466*t_2*l_1*p_1415040*l_1*p_10793664*t_2*l_2*p_1943040*l_2*p_42669319*t_2*l_2*p_2956800*l_1*p_85409964*t_2*l_1*p_2027520*l_1*p_96458484*l_2*p_11130240*t_1*l_2*p_1920*l_1*p_1296*l_1*p_4652160>")
\end{lstlisting}}

{\footnotesize
\begin{lstlisting}[breaklines=true,language=python,captionpos=b,texcl=false,frame=lines,caption={
Generators $x_7$, $y_7$, $x_4$, and $x_{14}$ for the maximal subgroup $7^2{:}\text{SL}_2(7)$ of $\MM$ discussed in Section~\ref{sec7^2}; see also \cite{ourfile_new}. The roles of the other elements are explained in Section~\ref{sec7^2}. (Note that $x_4$ is distinct from the elements $x_4$ in Listings~\ref{fig:secL211} and \ref{fig:sec11^2}.) \\}, label=fig:sec7^2]  
x7 = MM("M<y_5d3h*x_0a6dh*d_8d4h*p_111142481*l_1*p_2999040*l_1*p_43234193>")

y7 = MM("M<y_4a9h*x_1744h*d_0c88h*p_124439088*l_2*p_2597760*l_1*p_10860102*t_1*l_2*p_2386560*l_2*p_10772578*t_1*l_2*p_2956800*l_1*p_53817946*t_1*l_2*p_1858560*l_2*p_21333360*t_2*l_2*p_2830080*l_2*p_85837074*l_2*p_11151360*t_1*l_1*p_1457280*l_2*p_12549552*l_2>")

x4 = MM("M<y_406h*x_1bfeh*d_4d7h*p_44119992*l_2*p_2597760*l_1*p_33391058*l_2*t_1*l_1*p_2640000*l_1*p_12994999*l_1*t_1*l_2*p_2344320*l_2*p_1465428*l_1*t_1*l_1*p_2999040*l_1*p_5762*t_1*l_2*p_1985280*l_1*p_85413713*t_1*l_2*p_1943040*l_2*p_21367881*t_2*l_1*p_2027520*l_1*p_54866*t_1*l_1*p_1457280*l_2*p_76963>")

x14 = MM("M<y_599h*x_237h*d_0e76h*p_139011497*l_2*p_2956800*l_1*p_1912825*l_1*t_1*l_2*p_2787840*l_2*p_33397891*l_1*t_1*l_2*p_1393920*l_1*p_22416*l_2*p_2475840*t_2*l_2*p_2956800*l_1*p_10702214*t_2*l_1*p_2640000*l_1*p_661025*l_2*t_1*l_1*p_1457280*l_2*p_96458467*l_2*p_464640>")

a7 = MM("M<y_534h*x_144dh*d_4c9h*p_232106941*l_1*p_1415040*l_2*p_53793072*t_2*l_2*p_1900800*l_2*p_10674402*l_2*t_1*l_2*p_1900800*l_2*p_1928117*l_1*t_2*l_2*p_4371840*l_2*t_1*l_2*p_2956800*l_1*p_10693639*t_1*l_2*p_1943040*l_2*p_42668385*t_2*l_2*p_2597760*l_1*p_11141857*t_2*l_2*p_2830080*l_2*p_43151365>")

b7 = MM("M<y_4f2h*x_11c3h*d_322h*p_46570772*l_1*p_467520*l_1*p_11595168*l_2*t_1*l_1*p_2999040*l_1*p_21865425*l_1*t_1*l_1*p_1499520*l_1*p_1466374*l_1*t_1*l_1*p_1499520*l_2*p_85329058*t_2*l_1*p_1499520*l_2*p_43256279*t_1*l_2*p_1943040*l_2*p_64000626*t_2*l_1*p_1499520*l_2*p_43181204*t_1*l_2*p_2830080*l_2*p_64017988>")

c7 = MM("M<y_4dh*x_0bb5h*d_7a2h*p_11678766*l_2*p_1900800*l_2*p_25970*t_1*l_1*p_1499520*l_1*p_32552116*l_2*t_1*l_2*p_1920*l_2*p_23280*l_2*p_2819520*t_1*l_1*p_1457280*l_2*p_21441*t_2*l_1*p_1457280*l_2*p_96038950*l_1*p_3840*t_2*l_2*p_1943040*l_2*p_58663*l_2>")

d7 = MM("M<y_110h*x_0aaeh*d_67eh*p_97278295*l_2*p_2597760*l_1*p_13042995*l_1*t_2*l_1*p_2027520*l_1*p_33396803*l_2*t_1*l_1*p_1415040*l_2*p_24192*l_1*p_66240*t_1*l_2*p_2597760*l_1*p_32507814*t_1*l_1*p_1499520*l_2*p_43603527*t_1*l_2*p_2787840*l_2*p_494578*t_2*l_2*p_1985280*l_1*p_43161060>")

c = MM("M<y_477h*x_16e3h*d_0d02h*p_71006193*l_1*p_1457280*l_2*p_21804690*l_1*t_1*l_1*p_2999040*l_1*p_21812385*l_2*t_1*l_1*p_1499520*l_2*p_22755364*l_2*t_1*l_2*p_2386560*l_2*p_43604547*t_1*l_1*p_2880*l_1*p_1296*l_1*p_1548480*t_1*l_1*p_13326720*l_2*t_1*l_2*p_1943040*l_2*p_42715538>")

r = MM("M<y_185h*x_6e2h*d_89ah*p_124672224*l_2*p_43950720*l_1*p_71871360*l_1*t_2*l_1*p_21120*l_2*p_32461776*l_2*t_2*l_2*p_2787840*l_2*p_13053750*l_2*t_2*l_1*p_6547200*l_1*t_1*l_1*p_59917440*l_1*p_213375504*t_2*l_1*p_1499520*l_2*p_106666101*t_1*l_2*p_2830080*l_2*p_42715558*t_2*l_2*p_1985280*l_1*p_85410832>")

s = MM("M<y_7ch*x_8c6h*d_0cb8h*p_229806390*l_2*p_47942400*l_1*p_241760688*l_1*t_1*l_2*p_1943040*l_2*p_21865402*l_1*t_2*l_2*p_2956800*l_1*p_21893383*l_1*t_1*l_1*p_2640000*l_1*p_12502*t_2*l_1*p_1499520*l_1*p_64122852*t_2*l_2*p_3840*l_2*p_3120*l_2*p_6548160*t_1*l_2*p_2386560*l_2*p_42677985*t_1*l_2*p_1900800*l_2*p_214592>")

t = MM("M<y_183h*x_0cb9h*d_4b3h*p_162391248*l_1*p_50160000*l_1*p_222689328*l_1*t_2*l_2*p_1943040*l_2*p_32063514*l_1*t_1*l_2*p_1394880*l_2*p_22791936*l_1*t_1*l_1*p_1499520*l_1*p_10778215*t_2*l_2*p_1943040*l_2*p_85833293*t_1*l_1*p_1499520*l_2*p_43174440*t_1*l_2*p_1985280*l_1*p_43194643>")
\end{lstlisting}}

\vspace{11pt}
{\footnotesize
\begin{lstlisting}[breaklines=true,language=python,captionpos=b,texcl=false,frame=lines,caption={
Generators $x_2$ and $x_{19}$ for the maximal subgroup $\PSL_2(19){:}2$ of $\MM$ discussed in Section~\ref{secL219}; see also \cite{ourfile_new}. The roles of the other elements are explained in Section~\ref{secL219}. (Note that $x_2$ is distinct from the element $x_2$ in Listing~\ref{fig:secL211}.) \\}, label=fig:secL219]  
x2 = MM("M<y_531h*x_147bh*d_9e6h*p_236396641*l_2*p_1985280*l_1*p_21825835*l_1*t_2*l_1*p_1499520*l_2*p_32991824*l_2*t_2*l_2*p_1415040*l_1*p_24192*l_2*p_2496960*t_1*l_1*p_1499520*l_1*p_12113794*t_1*l_1*p_2027520*l_1*p_86261393*l_2*p_11594880*t_2*l_1*p_1394880*l_2*p_465792*l_1*p_2392320>")

x19 = MM("M<y_7ddh*x_0e39h*d_99h*p_136245180*l_1*p_1499520*l_1*p_23222870*t_1*l_1*p_1499520*l_2*p_1925316*l_1*t_2*l_1*p_2640000*l_1*p_1504850*l_1*t_1*l_2*p_2830080*l_2*p_127990691*t_1*l_2*p_1457280*l_1*p_175110*t_1*l_1*p_1920*l_2*p_10667712*l_2*p_6106560*t_1*l_2*p_2386560*l_2*p_42663561*t_2*l_2*p_2597760*l_1*p_85812125>")

y4 = MM("M<y_163h*x_1c92h*d_608h*p_59179108*l_1*p_2999040*l_1*p_43617991*t_1*l_2*p_2880*l_2*p_2160*l_2*p_2369280*t_2*l_2*p_1457280*l_1*p_1860708*l_2*t_2*l_2*p_2386560*l_2*p_85819731*t_1*l_2*p_1943040*l_2*p_42707029*t_2*l_2*p_2386560*l_2*p_21439072*t_1*l_1*p_1457280*l_2*p_10566>")

c = MM("M<y_0adh*x_128h*d_9bh*p_152633473*l_2*p_60804480*l_1*p_243091248*l_2*t_1*l_1*p_68787840*l_1*p_212044848*l_1*t_2*l_1*p_1499520*l_1*p_32535003*l_1*t_1*l_1*p_2027520*l_1*p_11521*t_1*l_2*p_783360*t_2*l_2*p_2830080*l_2*p_86257556*t_1*l_2*p_1943040*l_2*p_96477752*t_1*l_2*p_2956800*l_1*p_43197549>")

d = MM("M<y_4f1h*x_9bch*d_0f77h*p_106507260*l_1*p_80762880*l_2*p_213375504*t_2*l_1*p_1499520*l_2*p_583047*t_2*l_2*p_1900800*l_2*p_1040998*t_2*l_2*p_2386560*l_2*p_21331401*t_1>")
\end{lstlisting}}

\end{document}